\documentclass{amsart}
\usepackage{mathtools}
\usepackage{amsmath}
\usepackage{amssymb}
\usepackage{amsthm}
\usepackage{physics}
\usepackage{graphicx}

\usepackage[colorlinks]{hyperref}

\graphicspath{{fig/}}

\newtheorem{theorem}{Theorem}[section]

\theoremstyle{definition}

\newtheorem{example}[theorem]{Example}

\theoremstyle{remark}
\newtheorem{remark}[theorem]{Remark}

\numberwithin{equation}{section}

\begin{document}

\title{Complex harmonic mean}

\author{Atsushi~Nakayasu}
\address{Graduate School of Engineering, The University of Tokyo, Yayoi 2-11-16, Bunkyo–ku, Tokyo 113-8656, Japan}
% \curraddr{}
\email{ankys@g.ecc.u-tokyo.ac.jp}

% 47A64: Operator means involving linear operators, shorted linear operators, etc.
% 60A10: Probabilistic measure theory
% 30C10: Polynomials and rational functions of one complex variable

\date{\today}

%\dedicatory{}

%\commby{}

\begin{abstract}
We study the harmonic mean of non-zero complex-valued random variables (complex harmonic mean)
and establish several geometric estimates and bounds.
In contrast to the classical positive-valued case, complex harmonic means may lie outside the convex hull of the range.
We prove that if the range is contained in a disk not containing the origin,
then the complex harmonic mean is confined to the same disk.
This result is based on the behavior of disks under inversion and convexity arguments.
Further estimates involving the modulus and the real part are obtained, and the two-point case is analyzed explicitly, revealing a circular structure.
Several examples are provided to illustrate the distinctive features of complex harmonic means.
\end{abstract}

\maketitle

\section{Introduction}

In this short paper, we study the harmonic mean in the complex-valued setting
and show that it satisfies simple geometric bounds.

A \emph{(classical) harmonic mean} is usually defined for a positive-valued ($(0, \infty)$-valued) random variable $X$
and its definition is given by
\[
\mathbf{H}[X] := \mathbf{E}[X^{-1}]^{-1}
\]
if the expectation $\mathbf{E}[X^{-1}]$ exists as a positive number.
A \emph{complex harmonic mean} is defined for a non-zero complex-valued ($\mathbb{C}\setminus\{0\}$-valued) random variable $Z$ in a similar way
\[
\mathbf{H}[Z] := \mathbf{E}[Z^{-1}]^{-1}
\]
if the expectation $\mathbf{E}[Z^{-1}]$ exists as a non-zero complex number.
Obviously, if the range of a complex-valued random variable $Z$, which is the support of the distribution denoted by $\operatorname{Range}[Z]$, is contained in the positive numbers $(0, \infty)$, then the complex harmonic mean is nothing but the classical harmonic mean.
Note that a sufficient condition for the existence of a complex harmonic mean of $Z$ is that the range satisfies the condition
\[
\operatorname{Range}[Z] \subset \{ z \in \mathbb{C} \mid \Re z \ge a \}\cap\{ z \in \mathbb{C} \mid \abs{z} \le R \}
\]
for some $a > 0$ and $R < \infty$.

The harmonic mean has numerous applications,
and especially in the theory of homogenization, it appears as the Reuss bound of an effective coefficient \cite[Subsection 1.1.3]{r_a02}.
The complex harmonic mean should naturally appear when we consider homogenization of complex equations \cite{r_isf15}.
In fact, effective coefficients with positive real part in the complex setting have been well studied in the theory of viscoelastic composites \cite{r_gl93}, \cite{r_m02}.

For harmonic mean of a positive-valued random variable $X$, several bounds or estimates are known.
Trivial ones are the bounds by range:
\[
\inf \operatorname{Range}[X] \le \mathbf{H}[X] \le \sup \operatorname{Range}[X].
\]
(Here, the range $\operatorname{Range}[X]$ is a subset of positive numbers $(0, \infty)$.)
Another estimate is obtained from Jensen's inequality $\phi(\mathbf{E}[X]) \le \mathbf{E}[\phi(X)]$ for the convex function $\phi(x) = x^{-1}$ on $(0, \infty)$:
\[
\mathbf{H}[X] = \mathbf{E}[X^{-1}]^{-1} \le \mathbf{E}[(X^{-1})^{-1}] = \mathbf{E}[X].
\]
The equality holds if and only if $X$ is constant almost surely.

For complex harmonic mean, however, it is not trivial to derive bounds or estimates
and there are few results as far as the author knows.
% To the author’s knowledge, systematic bounds for harmonic means of complex-valued random variables have not been investigated.
We expect that such bounds are useful in studying effective coefficients for complex-valued equations.

This paper is organized as follows.
Section \ref{s_ex} provides several examples of complex harmonic means for two-point and continuous distributions,
illustrating their distinctive behavior compared with the positive-valued case with some figures.
Sections \ref{s_est_mod} and \ref{s_est_real} derive basic estimates involving the modulus and the real part.
Section \ref{s_bd_circ} contains our main geometric bound, showing that complex harmonic means inherit disk constraints from the range via inversion and convexity arguments.
Finally, in Section \ref{s_twoval} we revisit the two-point case and describe the resulting circular structure explicitly.

\section{Examples}
\label{s_ex}

We present several examples of complex harmonic means,
which illustrate that the complex harmonic mean behaves quite differently from the positive-valued one.

Consider the case when the range of a non-zero complex-valued random variable $Z$ consists of two values $c_1, c_2$, i.e. $\operatorname{Range}[Z] = \{ c_1, c_2 \}$
and let $\theta \in [0, 1]$ be the probability that $Z = c_2$.
Then, the complex harmonic mean $h = \mathbf{H}[Z]$ is given by
\[
h
= ((1-\theta) c_1^{-1}+\theta c_2^{-1})^{-1}
= \frac{c_1 c_2}{c_1\theta+c_2(1-\theta)}.
\]

\begin{example}
We now give the concrete example $c_1 = 1+i$, $c_2 = 1-i$.
Then, the complex harmonic mean $h = \mathbf{H}[Z]$ is calculated as
\[
h
= \frac{(1+i)(1-i)}{(1+i)\theta+(1-i)(1-\theta)}
= \frac{2}{1-(1-2\theta)i}
= \frac{1+(1-2\theta)i}{2\theta^2-2\theta+1}.
\]
In particular, for $\theta = 1/2$ we have $h = 2$.
This example shows that the real part of the complex harmonic mean may exceed the supremum of the real parts of the range.
We also remark that the complex harmonic mean $h$ is always on the circle
\[
\abs{z-1} = 1,
\]
which goes through $c_1$, $c_2$ and the origin $0$.
See Figure \ref{f_ex1}.
\end{example}

\begin{figure}
\includegraphics[width=0.75\textwidth]{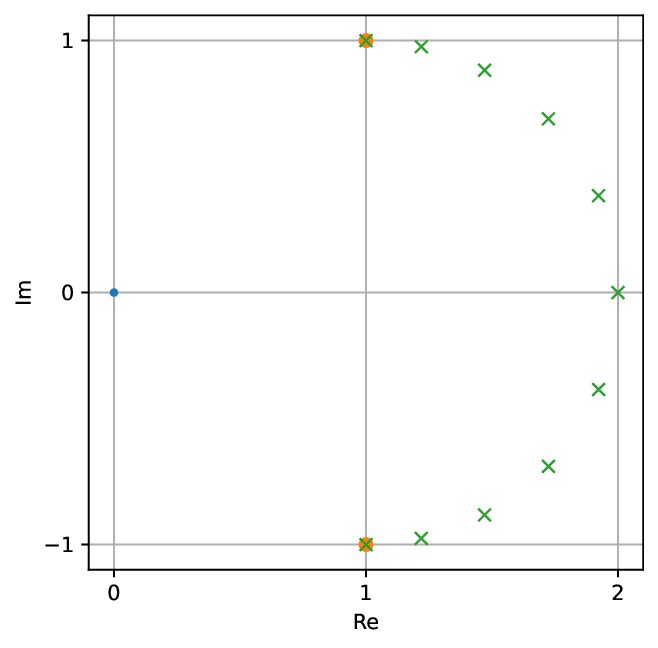}
\caption{Complex harmonic mean of $c_1 = 1+i$ and $c_2 = 1-i$ with the weight $\theta = 0.0, 0.1, \cdots, 1.0$.}
\label{f_ex1}
\end{figure}

\begin{example}
For another example, let $c_1 = 8$, $c_2 = 1+i$.
Then, the complex harmonic mean $h = \mathbf{H}[Z]$ is calculated as
\[
h
= \frac{8(1+i)}{8\theta+(1+i)(1-\theta)}
= \frac{8(1+3\theta)+32\theta i}{25\theta^2+6\theta+1}.
\]
Note that the imaginary part of the complex harmonic mean $(32\theta)/(25\theta^2+6\theta+1)$ attains its maximum at $\theta = 1/5$ and at this time $h = 4+2 i$.
This example shows that the imaginary part of the complex harmonic mean may exceed the supremum of the imaginary parts of the range.
We also remark that the complex harmonic mean $h$ is always on the circle
\[
\abs{z-(4-3 i)} = 5,
\]
which goes through $c_1$, $c_2$ and the origin $0$.
See Figure \ref{f_ex2}.
\end{example}

\begin{figure}
\includegraphics[width=0.9\textwidth]{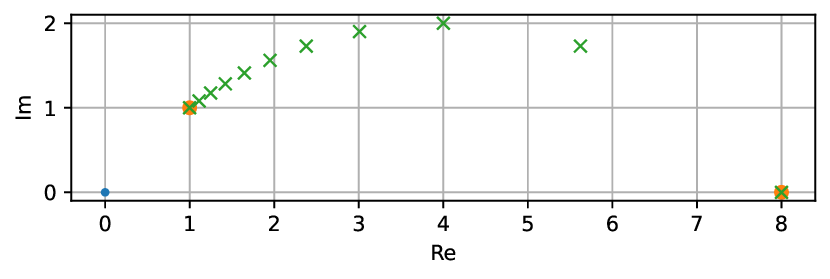}
\caption{Complex harmonic mean of $c_1 = 8$ and $c_2 = 1+i$ with the weight $\theta = 0.0, 0.1, \cdots, 1.0$.}
\label{f_ex2}
\end{figure}

We also study complex harmonic means for continuous distributions.

\begin{example}[Complex lognormal distribution]
Let $Z$ be a random variable with complex normal distribution $\mathcal{CN}(\mu, \sigma)$ with mean $\mu \in \mathbb{C}$ and variance $\sigma > 0$,
and consider the non-zero complex-valued random variable $\exp(Z)$ with lognormal distribution.
Note that the probability density function of $Z$ is given by
\[
f(z) = \frac{1}{\pi\sigma}\exp\qty(-\frac{\abs{z-\mu}^2}{\sigma}).
\]
Therefore, by a standard computation we have
\[
\begin{aligned}
\mathbf{E}[\exp(Z)]
&= \int_{\mathbb{C}} \exp(z)\frac{1}{\pi\sigma}\exp\qty(-\frac{\abs{z-\mu}^2}{\sigma})\dd{z}
= \frac{1}{\pi}e^{\mu}\int_{\mathbb{C}} \exp(-\abs{z}^2+\sqrt{\sigma}z)\dd{z} \\
&= \frac{1}{\pi}e^{\mu}\int_{\mathbb{R}} \exp(-x^2+\sqrt{\sigma}x)\dd{x}\int_{\mathbb{R}} \exp(-y^2+\sqrt{\sigma}y i)\dd{y} \\
&= \frac{1}{\pi}e^{\mu} \cdot \sqrt{\pi}e^{\frac{\sigma}{4}} \cdot \sqrt{\pi}e^{-\frac{\sigma}{4}}
= e^{\mu}.
\end{aligned}
\]
Similarly, we have
\[
\mathbf{H}[\exp(Z)]
= \mathbf{E}[(\exp(Z))^{-1}]^{-1}
= \mathbf{E}[\exp(-Z)]^{-1}
= \qty(e^{-\mu})^{-1}
= e^{\mu}.
\]
These computations show that the arithmetic mean $\mathbf{E}[\exp(Z)]$ and the harmonic mean $\mathbf{H}[\exp(Z)]$ can be the same for a complex lognormal distribution,
which never occurs for positive-valued settings unless trivial cases.
\end{example}

\section{Estimates by modulus}
\label{s_est_mod}

The first result in this paper is the estimate of complex harmonic mean in terms of the modulus:
\begin{equation}
\label{e_estabs}
\abs{\mathbf{H}[Z]} \ge \mathbf{H}[\abs{Z}].
\end{equation}

\begin{theorem}[Estimates by modulus]
For a non-zero complex-valued random variable $Z$, which has a complex harmonic mean $\mathbf{H}[Z]$,
the inequality \eqref{e_estabs} holds.
The equality holds when $Z = v X$ for some (constant) non-zero complex number $v$ and positive-valued random variable $X$.
\end{theorem}

\begin{proof}
Noting that
\[
Z^{-1} = \frac{\overline{Z}}{\abs{Z}^2},
\]
we have
\begin{equation}
\label{e_calc_chm}
\mathbf{H}[Z]
= \mathbf{E}\qty[\frac{\overline{Z}}{\abs{Z}^2}]^{-1}
= \frac{\overline{\mathbf{E}\qty[\frac{\overline{Z}}{\abs{Z}^2}]}}{\abs{\mathbf{E}\qty[\frac{\overline{Z}}{\abs{Z}^2}]}^2}
= \frac{\mathbf{E}\qty[\frac{Z}{\abs{Z}^2}]}{\abs{\mathbf{E}\qty[\frac{\overline{Z}}{\abs{Z}^2}]}^2}
\end{equation}
and hence
\[
\abs{\mathbf{H}[Z]}
= \frac{1}{\abs{\mathbf{E}\qty[\frac{Z}{\abs{Z}^2}]}}.
\]
Meanwhile,
\[
\mathbf{H}[\abs{Z}]
= \frac{1}{\mathbf{E}\qty[\frac{1}{\abs{Z}}]}.
\]
Now, in view of the triangle inequality, we have
\[
\abs{\mathbf{E}\qty[\frac{Z}{\abs{Z}^2}]}
\le \mathbf{E}\qty[\abs{\frac{Z}{\abs{Z}^2}}]
= \mathbf{E}\qty[\frac{1}{\abs{Z}}],
\]
which implies \eqref{e_estabs}.
The condition of equality follows from the equality condition in the triangle inequality.
\end{proof}

\section{Estimates by real part}
\label{s_est_real}

Next result of this paper is the estimate of complex harmonic mean by harmonic mean of real part:
\begin{equation}
\label{e:boundreal}
\Re \mathbf{H}[Z] \ge \mathbf{H}[\Re Z].
\end{equation}
We establish the following theorem with a small generalization.
For complex numbers $c$ and $z$ we define the inner product into real by $c\cdot z = \Re \bar{c}z$.

\begin{theorem}[Estimates by real part]
Let $c \in \mathbb{C}\setminus\{0\}$.
For a non-zero complex-valued random variable $Z$ with range
\[
\operatorname{Range}[Z] \subset \{ z \in \mathbb{C} \mid c\cdot z \ge a \}\cap\{ z \in \mathbb{C} \mid \abs{z} \le R \}
\]
with some $a > 0$ and $R < \infty$,
the inequality
\[
c\cdot \mathbf{H}[Z] \ge \mathbf{H}[c\cdot Z]
\]
holds.
The equality holds when $Z = v X$ for some (constant) non-zero complex number $v$ and positive-valued random variable $X$.
\end{theorem}

\begin{proof}
Since
\[
c\cdot \mathbf{H}[Z]
= \Re \overline{c}\mathbf{E}[Z^{-1}]^{-1}
% = \Re (\overline{c}^{-1}\mathbf{E}[Z^{-1}])^{-1}
= \Re \mathbf{E}[\overline{c}^{-1}Z^{-1}]^{-1}
= \Re \mathbf{E}[(\overline{c}Z)^{-1}]^{-1}
= \Re \mathbf{H}[\overline{c}Z],
\]
it is enough to show \eqref{e:boundreal}.
Thanks to \eqref{e_calc_chm},
we have
\[
\Re \mathbf{H}[Z]
= \frac{\mathbf{E}\qty[\frac{\Re Z}{\abs{Z}^2}]}{\mathbf{E}\qty[\frac{\Re Z}{\abs{Z}^2}]^2+\mathbf{E}\qty[\frac{\Im Z}{\abs{Z}^2}]^2}.
\]
Since
\[
\mathbf{H}[\Re Z]
= \frac{1}{\mathbf{E}\qty[\frac{1}{\Re Z}]} > 0,
\]
the sign of $\Re \mathbf{H}[Z]-\mathbf{H}[\Re Z]$ coincides with the one of
\[
I
= \mathbf{E}\qty[\frac{1}{\Re Z}]\mathbf{E}\qty[\frac{\Re Z}{\abs{Z}^2}]-\mathbf{E}\qty[\frac{\Re Z}{\abs{Z}^2}]^2-\mathbf{E}\qty[\frac{\Im Z}{\abs{Z}^2}]^2.
\]
Noting that
\[
\mathbf{E}\qty[\frac{1}{\Re Z}]-\mathbf{E}\qty[\frac{\Re Z}{\abs{Z}^2}]
= \mathbf{E}\qty[\frac{1}{\Re Z}-\frac{\Re Z}{\abs{Z}^2}]
= \mathbf{E}\qty[\frac{(\Im Z)^2}{\Re Z\abs{Z}^2}],
\]
we have
\[
I
= \mathbf{E}\qty[\frac{(\Im Z)^2}{\Re Z\abs{Z}^2}]\mathbf{E}\qty[\frac{\Re Z}{\abs{Z}^2}]-\mathbf{E}\qty[\frac{\Im Z}{\abs{Z}^2}]^2.
\]
Now, by applying the Hölder's inequality $\mathbf{E}[X Y]^2 \le \mathbf{E}[X^2]\mathbf{E}[Y^2]$, one will obtain
\[
I \ge \mathbf{E}\qty[\sqrt{\frac{(\Im Z)^2}{\abs{Z}^4}}]^2-\mathbf{E}\qty[\frac{\Im Z}{\abs{Z}^2}]^2 \ge 0,
\]
which implies $\Re \mathbf{H}[Z] \ge \mathbf{H}[\Re Z]$.

The equality holds when the real-valued random variables $\frac{(\Im Z)^2}{\Re Z\abs{Z}^2}$ and $\frac{\Re Z}{\abs{Z}^2}$ are linear dependent so that
\[
\frac{(\Im Z)^2}{\Re Z\abs{Z}^2} = k\frac{\Re Z}{\abs{Z}^2}
\]
for some $k \ge 0$.
Hence, $\Im Z = \pm\sqrt{k}\Re Z$ and so $Z = (1\pm\sqrt{k}i)\Re Z$.
\end{proof}

\section{Bounds by circles}
\label{s_bd_circ}

In this section we show bounds for complex harmonic means based on the range.
Our main observation is that inversion preserves disks \cite[Chapter VII, Section 5]{r_l99}, which leads to the following geometric bound.

\begin{theorem}[Bounds by circles]
Let $D$ be a disk not containing $0$ in the complex plane.
More precisely, let
\[
D = \{ z \in \mathbb{C} \mid \abs{z-c} \le r \}
\]
with $c \in \mathbb{C}$ and $r > 0$ satisfying $r \le \abs{c}$.
If a non-zero complex-valued random variable $Z$ satisfies
\[
\operatorname{Range}[Z] \subset D,
\]
then
\[
\mathbf{H}[Z] \in D
\]
holds.
\end{theorem}

A point of the proof is
that the inversion $z \mapsto z^{-1}$ maps a circle to a circle, a disk to a disk,
and that disks are convex.
We remark that the expectation of a complex-valued random variable whose range is contained in a closed convex set is in the same set.
See, e.g., \cite[Lemma 3.1]{r_n19}.

\begin{proof}
In the case $r < \abs{c}$, the inversion maps the boundary circle $C$ of $D$ to a circle $C'$ and $D$ to the closed disk $D'$ by $C$ since $0 \notin D$ (so $\infty \notin D'$).
Now, since $\operatorname{Range}[Z^{-1}] \subset D'$ and $D'$ is closed convex, we have $\mathbf{E}[Z^{-1}] \in D'$.
Again, by the inversion, we see that $\mathbf{H}[Z] = \mathbf{E}[Z^{-1}]^{-1} \in D$.

The case $r = \abs{c}$ is similar and a difference is just that the inversion maps $D$ to a closed half-plane.
Since a closed half-plane is convex,
one can see that $\mathbf{H}[Z] = \mathbf{E}[Z^{-1}]^{-1} \in D$ as well.
\end{proof}

\begin{remark}
We remark that the shape of $D$ in the theorem can be extended to images of closed convex sets $D'$ under inversion
such as intersections of disks or images of triangles.
\end{remark}

\section{Two-point cases revisited}
\label{s_twoval}

Consider the case when the range of $Z$ consists of two values $c_1 \ne c_2$, $c_1, c_2 \ne 0$, i.e.\ $\operatorname{Range}[Z] = \{ c_1, c_2 \}$.
In this case, $\operatorname{Range}[Z^{-1}] = \{ c_1^{-1}, c_2^{-1} \}$
and the expectation $\mathbf{E}[Z^{-1}]$ moves along the line segment between $c_1^{-1}$ and $c_2^{-1}$.
Since the inversion maps a line to a circle going through $0$,
we have the following result.

\begin{theorem}[Two-point cases]
Let $c_1, c_2 \ne 0$ satisfy $c_1 \ne c_2$.
Let a non-zero complex-valued random variable $Z$ satisfy
\[
\operatorname{Range}[Z] = \{ c_1, c_2 \}.
\]
\begin{itemize}
\item
If $c_1$, $c_2$, $0$ are not on a common line,
then the complex harmonic mean $\mathbf{H}[Z]$ moves along the arc of the circle passing through $c_1$, $c_2$ and $0$ (uniquely determined) between $c_1$ and $c_2$ without $0$.
\item
If $c_1$, $c_2$, $0$ are on a common line and $0$ is not on the line segment between $c_1$ and $c_2$,
then the complex harmonic mean $\mathbf{H}[Z]$ moves along the line segment between $c_1$ and $c_2$.
\end{itemize}
\end{theorem}

\section*{Acknowledgements}
The author thanks Hiroaki~Deguchi and Takayuki~Yamada for fruitful discussions and remarkable notes on the theory of viscoelastic composites.
This work was supported by JSPS KAKENHI Grant Number JP25K17275.

\end{document}